\documentclass{amsart}
\usepackage{amsfonts, color}
\usepackage{latexsym}
\usepackage{amssymb}
\usepackage{amsmath}

\newcommand{\R}{\mathbb R}
\newcommand{\N}{\mathbb N}

\newcommand{\E}{\mathbb E}
\newcommand{\Pro}{\mathbb P}

\newtheorem{thm}{Theorem}[section]
\newtheorem{cor}[thm]{Corollary}
\newtheorem{lemma}[thm]{Lemma}
\newtheorem{df}[thm]{Definition}
\newtheorem{proposition}[thm]{Proposition}
\newtheorem{conjecture}[thm]{Conjecture}
\newtheorem*{rmk}{Remark}
\DeclareMathOperator{\signum}{sgn}
\newcommand{\Vol}{\textup{Vol}}
\begin{document}

%%%%%%%%%%%%%%%%%%%%%%%%%%%%%%%%%%%%%%%%%%%%%5

\title{The variance conjecture on some polytopes}

\author[D.\,Alonso]{David Alonso-Guti\'errez}
\address{Universidad de Murcia}
\email{davidalonso@um.es}
\thanks{ The first named author is partially supported by
 Spanish grants MTM2010-16679 and
 \lq\lq Programa de Ayudas a Grupos de Excelencia de
 la Regi\'on de Murcia\rq\rq, Fundaci\'on S\'eneca, 04540/GERM/06.}

\author[J.\,Bastero]{Jes\'us Bastero}
\address{Universidad de Zaragoza}\email[(Jes\'us
Bastero)]{bastero@unizar.es}
\thanks{The second author is partially supported by Spanish grant MTM2010-16679}

\begin{abstract}
We show that any random vector uniformly distributed on any hyperplane
projection of $B_1^n$ or $B_\infty^n$ verifies the variance conjecture
$$\text{Var }|X|^2\leq C\sup_{\xi\in S^{n-1}}\E\langle X,\xi\rangle^2\E|X|^2.$$
Furthermore, a random vector uniformly distributed on a hyperplane projection of
$B_\infty^n$ verifies a negative square correlation property and consequently
any of its linear images verifies the variance conjecture.
\end{abstract}

\date{\today}

\maketitle

\section{Introduction and notation}

Let $X$ be a random vector in $\R^n $ with a log-concave density, {\it i.e.} $X$
is
distributed on $\R^ n$ according
to a probability measure, $\mu_X$, whose density with respect to the Lebesgue
measure is  $\exp(-V)$
for some convex function $V:\R^n\to (-\infty,\infty]$. For instance, vectors
uniformly
distributed on
convex bodies and Gaussian random vectors are log-concave.

A random vector $X$ is said to be isotropic if:
\begin{enumerate}
\item[i)]{The barycenter is at the origin, {\it i.e.}, $\E X=0$, and}

\item[ii)]{The covariance matrix $M_X$ is the identity $I_n$, {\it i.e.} $\E
\langle X,e_i\rangle\langle X,e_j\rangle
=\delta_{i,j}$, $1\leq i,j\le n$,}
\end{enumerate}
 where $\{e_i\}_{i=1}^n$ denotes the canonical basis in $\R^n$ and
$\delta_{i,j}$ denotes the Kronecker delta. It is well known that every centered
random
vector with full dimensional support
has a unique, up to orthogonal
transformations, linear image  which is isotropic.

Given a log-concave random vector $X$, we will denote by $\lambda_X^2$ the highest
eigenvalue of the
covariance matrix $M_X$ and by $\sigma_X$ its \lq\lq thin shell width\rq\rq\. {\it i.e.}
\begin{eqnarray*} \lambda_X^2&=&\Vert M_X\Vert_{\ell
^n_2\to \ell^n_2}=\sup_{\xi\in S^{n-1}}\E\langle X,\xi\rangle^2,\cr
\sigma_X&=&\sqrt{\E\left\vert|X|-(\E|X|^2)^\frac{1}{2}\,\right\vert^2}.
\end{eqnarray*}
($S^{n-1}$ represents the Euclidean unit sphere in $\R^n$).

In  Asymptotic Geometric Analysis, the {\it variance }
conjecture states the following:
\begin{conjecture}\label{varianceiso}
 There exists an absolute constant $C$ such that for every isotropic log-concave
 vector $X$, if we denote by $|X|$ its Euclidean norm,
$$
  \text{Var }|X|^2\leq C\E|X|^2= Cn.
$$
\end{conjecture}
This conjecture was considered by
Bobkov and Koldobsky in the context
of the Central Limit Problem for isotropic convex bodies  (see \cite{BK}).
It was conjectured before by Antilla, Ball and Perissinaki,
(see \cite{ABP}) that for an isotropic log-concave vector $X$,
 $|X|$ is highly concentrated in a \lq\lq thin shell\rq\rq\,
more than the trivial bound $\text{Var}|X|\leq \E|X|^2$ suggested. Actually, it
is known that the variance conjecture is equivalent to the {\it
thin shell width} conjecture:
\begin{conjecture}\label{ThinShellWidthConjecture}
There exists an absolute constant $C$ such that for every isotropic log-concave
vector $X$
$$
 \sigma_X=\sqrt{\E\left\vert|X|-\sqrt n\,\right\vert^2}\,\leq C.
$$
\end{conjecture}

It is also known (see \cite{BN},
\cite{EK}) that these two equivalent conjectures are stronger than the
hyperplane conjecture, which states that every convex body of volume 1 has a
hyperplane section of volume greater than some absolute constant.

The variance conjecture is a particular case of a stronger conjecture,  due to
Kannan, Lov\'asz and Simonovits (see \cite{KLS}) concerning the
spectral gap
of log-concave probabilities. This conjecture  can be stated in the following way
due to the work of Cheeger, Maz'ya and Ledoux:
\begin{conjecture}\label{KLS}
 There exists an absolute constant $C$ such that for any locally Lipschitz
function,
$g:\R^n\to \R$,
 and any centered log-concave random vector $X$ in $\R^n$
$$
 \text{Var }g(X)\leq C\lambda_X^2\E|\nabla g(X)|^2.
$$
\end{conjecture}

Note that Conjecture \ref{varianceiso} is the particular case of Conjecture
\ref{KLS}, when we consider only isotropic vectors and $g(X)=|X|^2$. Our purpose
in this paper is to study the particular case of Conjecture \ref{KLS} in which
$g(X)=|X|^2$ but $X$ is not necessarily isotropic. Thus, we will study the
following {\it general variance} conjecture:
\begin{conjecture}\label{variance}
 There exists an absolute constant $C$ such that for every centered log-concave
 vector $X$
$$
  \text{Var }|X|^2\leq C\lambda_X^2\E|X|^2.
$$
\end{conjecture}

In the same way that Conjecture \ref{varianceiso} is equivalent to Conjecture
 \ref{ThinShellWidthConjecture}, Conjecture \ref{variance} can be shown
(see Section \ref{General}) to be equivalent to the following {\it general thin shell width} conjecture:
\begin{conjecture}\label{ThinShellGeneral}
There exists an absolute constant $C$ such that for every centered log-concave
 vector $X$
$$
\sigma_X\leq C\lambda_X
$$
\end{conjecture}

Notice that since Conjecture \ref{variance} and Conjecture
\ref{ThinShellGeneral} are not invariant under linear maps,
these two conjectures cannot easily be reduced to Conjecture \ref{varianceiso}
 and Conjecture \ref{ThinShellWidthConjecture}. We
will study how these conjectures behave under linear transformations and we will
also prove that random vectors uniformly distributed on a certain family of
polytopes verify Conjecture \ref{variance} but, before stating
our results, let us recall
the results known, up to now, concerning the aformentioned conjectures.

Besides the Gaussian vectors only a few examples  are known to satisfy
Conjecture \ref{KLS}. For instance, the vectors uniformly distributed on
$\ell^n_p$-balls,
$1\leq p\leq
\infty$, the simplex and some revolution convex bodies (\cite{S}, \cite{LW},
\cite{BW}, \cite{H}). In \cite{K4},
Klartag proved Conjecture \ref{KLS} with an
extra  $\log n$ factor for vectors uniformly distributed on unconditional convex
bodies and
recently Barthe and Cordero extended this result for log-concave vectors with
many
symmetries
(see \cite{BC}).  Kannan, Lov\'asz and Simonovits proved Conjecture \ref{KLS}
with the factor
$(\E |X|)^2$  instead of $\lambda_X^2$ (see \cite{KLS}), improved by Bobkov to
$(\text{Var}|X|^2)^{1/2}\simeq
\sigma_X\,\E |X|$ (see
\cite{B3}).

In \cite{K3}, Klartag proved Conjecture \ref{varianceiso}
for random vectors uniformly distributed on isotropic unconditional convex
bodies. The best known (dimension dependent) bound for general log-concave
isotropic random vectors in Conjecture \ref{ThinShellWidthConjecture}
was proved by Gu\'{e}don and Milman  with a factor
$n^{1/3}$ instead of  $C$,
improving down
to $n^{1/4}$ when $X$ is $\psi_2$ (see, \cite{GM}). This results give better
estimates than previous  ones by Klartag (see
\cite{K2}) and Fleury (see \cite{F}).
Given the relation between Conjecture \ref{varianceiso} and Conjecture
{\ref{ThinShellWidthConjecture} we have that Conjecture \ref{varianceiso} is
known to be true
with an
extra factor
$n^{2/3}$.

Very recently  Eldan, (\cite{E}) obtained a
breakthrough showing that Conjecture \ref{ThinShellWidthConjecture} implies
Conjecture \ref{KLS} with an extra logarithmic factor. By using the result of
Gu\'edon-Milman,
Conjecture \ref{KLS} is obtained with an extra factor $n^{2/3}(\log n)^2 $.

Since the variance conjecture is not linearly invariant, in Section
\ref{General} we will study its behavior under linear transformations {\it
i.e.}, given a centered log-concave random vector $X$, we will study the
variance conjecture of the random vector $TX$, $T\in GL(n)$. We will prove that
if $X$ is an isotropic random vector verifying Conjecture \ref{varianceiso},
then the non-isotropic $T\circ U(X) $ verifies  the variance
 conjecture (\ref{variance})
for a typical $U\in O(n)$. We will also show
the equivalence between Conjecture \ref{variance} and Conjecture \ref{ThinShellGeneral}.
As a consequence of Gu\'edon and Milman's result we
obtain that every centered log-concave random vector verifies the variance
conjecture with constant $Cn^\frac{2}{3}$ rather than the
 $Cn^\frac{2}{3}(\log n)^2$
obtained from the best general known result in Conjecture \ref{KLS}.

The main results in this paper will be included in Section \ref{projections},
where we will show that random vectors uniformly disitributed on hyperplane
projections of $B_1^n$ or $B_\infty^n$ (the unit balls of $\ell_1^n$ and
$\ell_\infty^n$ respectively) verify Conjecture \ref{variance}. Furthermore, in
the case of the hyperplane projections of $B_\infty^n$ we will see that they
verify a negative square correlation property with respect to any orthonormal
basis, which will allow us to deduce that also a random vector uniformly
distributed on any linear image of any hyperplane projection of $B_\infty^n$
will verify Conjecture \ref{variance}.

In order to compute some quantities on the hyperplane projections of $B_1^n$ and
$B_\infty^n$ we will use Cauchy's formula which, in the case of polytopes can be
stated like this:

Let $K_0$ be a polytope with facets $\{F_i\,:\, i\in I\}$ and $K=P_HK_0$ be the
projection of $K_0$ onto a hyperplane. If $X$ is a random vector uniformly
distributed on $K$, for any integrable function $f:K\to\R$ we have
$$
\E f(X)=\sum_{i\in I}\frac{\Vol(P_H(F_i))}{\Vol(K)}\E f(P_H Y^i),
$$
where $Y^i$ is a random vector uniformly distributed on the facet $F_i$ and
$\Vol$ denotes the volume or Lebesgue measure.

Let us now introduce some notation. Given a convex body $K$, we will denote by
$\widetilde{K}$ its homothetic image of volume 1 ($\Vol(\widetilde{K})=1$).
$\widetilde{K}=\frac{K}{\Vol^\frac{1}{n}(K)}$.
We recall that a convex body $K\subset\R^n$ is \emph{isotropic} if
it has volume $\Vol(K)=1$, the barycenter of $K$ is at the origin and its
inertia matrix is a multiple
of the identity. Equivalently, there exists a constant $L_K>0$
called isotropy constant of $K$ such that $L_K^2=\int_K\langle
x,\theta\rangle^2\, dx, \forall \theta\in S^{n-1}$. In this case
if $X$ denotes a random vector uniformly distributed on $K$, $\lambda_X=L_K$.
Thus, $K$ is
isotropic if the random vector $X$, distributed on $L_K^{-1}K$ with density
$ L^n_K \chi_{L_K^{-1}K}$
is isotropic.

When we write $a\sim b$,
for $a,b>0$, it means that the quotient of $a$ and $b$ is bounded from above and
from bellow by
absolute constants. $O(n)$ will always denote the orthogonal group on $\R^n$.
%notacion de base ortonormal

%%%%%%%%%%%%%%%%%%%%%%%%%%%%%%%%%%%%%%%%%%%%%%%%%%%%%%%%%%%%%%%%%%%%%%%%%%%%%%%%
\section{General results}\label{General}
%%%%%%%%%%%%%%%%%%%%%%%%%%%%%%%%%%%%%%%%%%%%%%%%%%%%%%%%%%%%%%%%%%%%%%%%%%%%%%%

In this section we are going to consider the variance conjecture for linea\cite{K1},r
transformatios of a fixed centered log-concave random vector in $\R^n$. Our
first result shows that if such random vector is
not far from being isotropic and verifies the variance conjecture, then the
average perturbation (in the sense we will state in the proposition) will also
verify the variance conjecture.

 \begin{proposition}\label{PropositionAverage}
Let $X$ be a centered isotropic, log-concave random vector in $\R^n$ verifying the
variance conjecture with constant $C_1$. Let $T\in GL(n)$ be any linear transformation.
If  $U$ is a random map uniformly distributed in $O(n)$ then
\[
 \E_U\text{Var } |T\circ U(X)|^2\leq CC_1\Vert T\Vert_{op}^2\Vert T\Vert_{HS}^2=
CC_1\lambda_{T\circ u(X)}^2\E|T\circ u(X)|^2
\]
for any $u\in O(n)$,
where $C$ is an absolute constant.
 \end{proposition}
\begin{proof} The non singular linear map $T$ can be expressed
by $T=V\Lambda U_1$ where
$V, U_1\in O(n)$  and $\Lambda=[\lambda_1,\dots,
\lambda_n]$ ( $\lambda_i>0$) a diagonal map.\\
\indent Given $\{e_i\}_{i=1}^n$ the canonical basis in $\R^n$, we will identify every
$U\in O(n)$ with the orthonormal basis $\{\eta_i\}_{i=1}^n$ such that
$U_1U\eta_i=e_i$ for all $i$.
Thus, by uniqueness of the Haar measure invariant under the
action of $O(n)$ we have that, for any integrable function $F$
\begin{eqnarray*}
\E_UF(U)&=&\E_U
F(\eta_1,\dots,\eta_n)\cr
&=&\int_{S^{n-1}}\int_{S^{n-1}\cap\eta_1^\perp}\dots\int_{S^
{n-1}\cap\eta_1^\perp\cap\dots\cap\eta_{n-1}^\perp}F(\eta_1,\dots,\eta_n)d\nu(\eta_
n)\dots d\nu(\eta_1),
\end{eqnarray*}
where $d\nu(\eta_i)$ is the Haar probability measure on
$S^{n-1}\cap\eta_1^\perp\cap\dots\cap\eta_{i-1}^\perp$.
Then, since $\displaystyle{|T\circ U(X)|^2=|\Lambda U_1 UX|^2=\sum_{i=1}^n\langle \Lambda
U_1UX,e_i\rangle^2=\sum_{i=1}^n\lambda_i^2\langle X,\eta_i\rangle^2}$
\begin{eqnarray*}
 \E_U\text{Var } |T\circ U(X)|^2&=&\sum_{i=1}^n\lambda_i^4\E_U(\E\langle
X,\eta_i\rangle^4-(\E\langle X,\eta_i\rangle^2)^2)\cr
&+&\sum_{i\neq j}\lambda_i^2\lambda_j^2\E_U(\E\langle X,\eta_i\rangle^2\langle
X,\eta_j\rangle^2-\E\langle X,\eta_i\rangle^2\E\langle X,\eta_j\rangle^2).
\end{eqnarray*}

Since for every $i$
\begin{eqnarray*}
\E_U(\E\langle
X,\eta_i\rangle^4-(\E\langle X,\eta_i\rangle^2)^2)&\leq& \E_U\E\langle
X,\eta_i\rangle^4
=\E |X|^4 \int_{S^{n-1}}\langle e_1,\eta_1\rangle^4d\nu(\eta_1)\cr
&=&\frac{3}{n(n+2)}\E |X|^4,
\end{eqnarray*}
and for every $i\neq j$, denoting by $Y$ an independent copy of $X$,
\begin{eqnarray*}
&&\E_U(\E\langle X,\eta_i\rangle^2\langle
X,\eta_j\rangle^2-\E\langle X,\eta_i\rangle^2\E\langle X,\eta_j\rangle^2)\cr
&&=\E|X|^4\int_{S^{n-1}}\left\langle
\frac{X}{|X|},\eta_1\right\rangle^2\int_{S^{n-1}\cap\eta_1^\perp}\left\langle
\frac{X}{|X|},\eta_2\right\rangle^2d\nu(\eta_2)d\nu(\eta_1)\cr
&&-\E|X|^2|Y|^2\int_{S^{n-1}}\left\langle\frac{X}{|X|},\eta_1\right\rangle^2\int_
{ S^{n-1}\cap\eta_1^\perp}\left\langle\frac{Y}{|Y|},
\eta_2\right\rangle^2d\nu(\eta_2)d\nu(\eta_1)\cr
&&=\frac{\E|X|^4}{n-1}\left(\frac{1}{n}-\int_{S^{n-1}}\langle
e_1,\eta_1\rangle^4d\nu(\eta_1)\right)\cr
&&-\frac{\E|X|^2|Y|^2}{n-1}\left(\frac{1}{n}-\int_{S^{n-1}}\left\langle\frac{X}{
|X|},\eta_1\right\rangle^2\left\langle\frac{Y}{|Y|},
\eta_1\right\rangle^2d\nu(\eta_1)\right)\cr
&&=\frac{\E|X|^4}{n-1}\left(\frac{1}{n}-\frac{3}{n(n+2)}\right)\cr
&&-\frac{\E|X|^2|Y|^2}{n-1}\left(\frac{1}{n}-\frac{1}{n(n+2)}-\frac{2}{n(n+2)}
\left\langle\frac{X}{
|X|},\frac{Y}{|Y|}\right\rangle^2\right)\cr
\end{eqnarray*}
we have that
\begin{eqnarray*}
 &&\E_U\text{Var } |T\circ U(X)|^2\leq\frac{3}{n(n+2)}\E |X|^4\sum_{i=1}^n\lambda_i^4\cr
&&+\left(\frac{\E|X|^4-(\E|X|^2)^2}{n(n+2)}-\frac{2\E|X|^2|Y|^2}{(n-1)n(n+2)}
\left(1-\left\langle\frac{X}{
|X|},\frac{Y}{|Y|}\right\rangle^2\right)\right)\sum_{i\neq
j}\lambda_i^2\lambda_j^2\cr
&&\leq\frac{3}{n(n+2)}\E
|X|^4\sum_{i=1}^n\lambda_i^4+\frac{\text{Var }|X|^2}{n(n+2)}\sum_{i\neq
j}\lambda_i^2\lambda_j^2.
\end{eqnarray*}

Now, since for every $\theta\in S^{n-1}$,
$
 \E\langle X,\theta\rangle^2=1
$ and $X$ satisfies the variance conjecture with constant $C_1$,
we have
$$
\E|X|^4=Var|X|^2+(\E|X|^2)^2\leq C_1n+n^2\leq CC_1 n^2.
$$ and
\begin{eqnarray*}
\E_U\text{Var } |T\circ U(X)|^2&\leq& C C_1\sum_{i=1}^n\lambda_i^4+
\frac{C_1}{n}\sum_{i\neq j}\lambda_i^2\lambda_j^2\cr
\end{eqnarray*}
Hence, given any $u\in O(n)$, let $\{\nu_i\}_{i=1}^n$ be
the orthonormal basis defined by $\nu_i=U_1\circ u(e_1)$, for all $ i$. Then
we have
\begin{eqnarray*}
 \lambda_{T\circ u(X)}^2&=&\sup_{\theta\in S^{n-1}}\E\langle
T\circ u(X),\theta\rangle^2=\sup_{\theta\in S^{n-1}}\E\langle \Lambda U_1
uX,\theta\rangle^2\\
&=&\sup_{\theta\in S^{n-1}}\E\left(\sum_{i=1}^n\lambda_i\langle X,\nu_i\rangle
\langle e_i,\theta\rangle\right)^2\\
&=&\sup_{\theta\in S^{n-1}}\sum_{i=1}^n\lambda_i^2\langle e_i, \theta\rangle^2
=
\max_{1\leq i\leq n}\lambda_i^2=\Vert T\Vert_{op}^2
\end{eqnarray*}and
\begin{eqnarray*}
 \E|T\circ u(X)|^2&=&\sum_{i=1}^n\lambda_i^2\E\langle X,\nu_i\rangle^2
=\sum_{i=1}^n\lambda_i^2=\Vert T\Vert^2_{HS}
\end{eqnarray*}
Thus
\begin{eqnarray*}
\E_U\text{Var } |T\circ U(X)|^2
&\leq&C C_1\Vert T\Vert_{op}^2\Vert T\Vert^2_{HS} +
\frac{C_1}{n}\Vert T\Vert_{op}^2\sum_{i\neq j}\lambda_j^2\cr
&\leq&C C_1\Vert T\Vert_{op}^2\Vert T\Vert^2_{HS} +
C_1\Vert T\Vert_{op}^2\Vert T\Vert^2_{HS}\\
&=&CC_1\Vert T\Vert_{op}^2\Vert T\Vert_{HS}^2=
CC_1\lambda_{T\circ u(X)}^2\E|T\circ u(X)|^2 \\
\end{eqnarray*}
\end{proof}

\begin{rmk}
 The same proof as before can be applied when $X$ is not necessarily isotropic. In this case
$$
\E_U\text{Var } |T\circ U(X)|^2\leq CC_1\lambda_{T\circ u(X)}^2\E|T\circ u(X)|^2
$$
for any $u\in O(n)$, where
 $B$ is the
spectral condition number of its covariance matrix {\it i.e.}
$$\displaystyle{B^2=\frac{\max_{\theta\in S^{n-1}}\E\langle X,\theta\rangle^2}{\min_{\theta\in
S^{n-1}}\E\langle X,\theta\rangle^2}}.$$
\end{rmk}

As a consequence of Markov's inequality we obtain the following
\begin{cor}
 Let $X$ be an isotropic, log-concave random vector in $\R^n$ verifying
the variance conjecture with constant $C_1$. There exists an absolut constant
$C$ such that  the measure of the
set of  orthogonal operators $U$ for which the random vector
$T\circ U(X)$ verifies the variance conjecture with constant $CC_1$ is
greater than $\frac{1}{2}$.

\end{cor}

In \cite{GM} it was shown that every log-concave isotropic random vector
verifies the thin-shell width conjecture with constant $C_1=Cn^\frac{1}{3}$. Also,
an estimate for $\sigma_X$ was given when  $X$ is not isotropic.

The following proposition is well known for the experts. However we include here
for the sake of completeness.  As a consequence and, by using
 the result in \cite{GM}, we will obtain that every centered
log-concave vector, isotropic or not, verifies the variance conjecture
 with constant
$Cn^\frac{2}{3}$ rather than $Cn^\frac{2}{3}(\log n)^2$.

\begin{proposition}
 Let $X$ be an isotropic log-concave random vector, $T$ a linear map and
$\sigma_{TX}$ the thin-shell width of the random vector $TX$ {\it i.e.}
$$
\sigma_{TX}^2=\E\left||TX|-(\E|TX|^2)^\frac{1}{2}\right|^2.
$$
Then
$$
\sigma_{TX}^2\leq\frac{\text{Var }|TX|^2}{\E|TX|^2}\leq C_1\sigma_{TX}^2+
C_2\frac{\Vert
T\Vert_{op}^2}{\Vert
T\Vert_{HS}^2}\lambda_{TX}^2.
$$
\end{proposition}

\begin{proof}
The first inequality is clear, since
\begin{eqnarray*}
 \sigma_{TX}^2&=&\E\left||TX|-(\E|TX|^2)^\frac{1}{2}\right|^2\leq
\E\left||TX|-(\E|TX|^2)^\frac{1}{2}\right|^2\frac{\left||TX|+(\E|TX|^2)^\frac{1}
{2}\right|^2}{\E|TX|^2}\cr
&=&\frac{\text{Var }|TX|^2}{\E|TX|^2}.
\end{eqnarray*}
Let us now show the second inequality. Let $B>0$ to be chosen later.
\begin{eqnarray*}
 \text{Var }|TX|^2&=&\E\left||TX|^2-\E|TX|^2\right|^2\chi_{\left\{|TX|\leq
B(\E|TX|^2)^\frac{1}{2}\right\}}\cr
&+&\E\left||TX|^2-\E|TX|^2\right|^2\chi_{\left\{|TX|>
B(\E|TX|^2)^\frac{1}{2}\right\}}.\cr
\end{eqnarray*}
The first term equals
\begin{eqnarray*}
&&\E\left||TX|+(\E|TX|^2)^\frac{1}{2}\right|^2\left||TX|-(\E|TX|^2)^\frac{1}{2
} \right|^2\chi_{\left\{|TX|\leq B(\E|TX|^2)^\frac{1}{2}\right\}}\cr
&&\leq(1+B)^2\sigma_{TX}^2\E|TX|^2.
\end{eqnarray*}
If $B\geq\frac{1}{\sqrt 2}$, the second term verifies
\begin{eqnarray*}
&&\E\left||TX|^2-\E|TX|^2\right|^2\chi_{\left\{|TX|>
B(\E|TX|^2)^\frac{1}{2}\right\}}\leq\E|TX|^4\chi_{\left\{|TX|>
B(\E|TX|^2)^\frac{1}{2}\right\}}\cr
\end{eqnarray*}
By Paouris' strong estimate for log-concave isotropic probabilities (see
\cite{Pa}) there exists an absolute constant $c$ such that
$$
\Pro\{|TX|>ct(\E|TX|^2)^\frac{1}{2}\}\leq e^{-t\frac{\Vert T\Vert_{HS}}{\Vert
T\Vert_{op}}}\hspace{1cm}\forall t\geq 1.
$$
Choosing $B=\max\left\{c,\frac{1}{\sqrt 2}\right\}$ we have that the
second term is bounded from above by
\begin{eqnarray*}
 &&\E|TX|^4\chi_{\{|TX|>B(\E|TX|^2)^\frac{1}{2}\}}=
B^4(\E|TX|^2)^2\Pro\{|TX|>B(\E|TX|^2)^\frac{1}{2}\}\cr
&&+B^4(\E|TX|^2)^2\int_1^\infty 4t^3\Pro\{|TX|>Bt(\E|TX|^2)^\frac{1}{2}\}dt\cr
&&\leq B^4\Vert T\Vert_{op}^4\frac{\Vert T\Vert_{HS}^4}{\Vert
T\Vert_{op}^4}e^{-\frac{\Vert T\Vert_{HS}}{\Vert T\Vert_{op}}}+ B^4\Vert
T\Vert_{HS}^4\int_1^\infty4t^3e^{-t\frac{\Vert T\Vert_{HS}}{\Vert
T\Vert_{op}}}dt\cr
&&\leq C_2\Vert T\Vert_{op}^4.
\end{eqnarray*}

Hence, we achieve
\begin{eqnarray*}
\frac{\text{Var }|TX|^2}{\E|TX|^2}&\leq&\sigma_{TX}^2\left(1+B\right)^2+
C_2\frac{\Vert
T\Vert_{op}^4}{\Vert
T\Vert_{HS}^2}\cr
&\leq&C_1\sigma_{TX}^2+ C_2\frac{\Vert T\Vert_{op}^4}{\Vert
T\Vert_{HS}^2}\cr
&=& C_1\sigma_{TX}^2+ C_2\frac{\Vert T\Vert_{op}^2}{\Vert
T\Vert_{HS}^2}\lambda_{TX}^2.
\end{eqnarray*}

\end{proof}
As a consequence of this proposition we obtain that Conjecture
\ref{variance} and Conjecture \ref{ThinShellGeneral} are equivalent.
Combining it with the estimate of
$\sigma_{TX}$ given in \cite{GM} we obtain the following
\begin{cor}
 There exists an absolute constant $C$ such that
 for every log-concave isotropic
random vector $X$ and any linear map $T\, \in GL(n)$ we have
$$\sigma_{TX}\leq C_1\lambda_{TX}\qquad\Longrightarrow\enspace
\text{Var }|TX|^2\leq C_1\E|TX|^2$$
and
$$\text{Var }|TX|^2\leq C_2\E|TX|^2\qquad\Longrightarrow\enspace
\sigma_{TX}\leq C\,C_2\lambda_{TX}.$$
Moreover, both inequalities are true with $C_2=Cn^{2/3}$.

\end{cor}
\begin{proof} The two implications are a direct consequence of the
previous proposition and the fact that $\Vert T\Vert_{op}\leq\Vert T\Vert_{HS}$.
 It was proved in \cite{GM} that
$$
\sigma_{TX}\leq C\Vert T\Vert_{op}^\frac{1}{3}\Vert T\Vert_{HS}^\frac{2}{3}.
$$
Thus, by the previous proposition
\begin{eqnarray*}
 \frac{\text{Var }|TX|^2}{\E|TX|^2}&\leq&C_1\sigma_{TX}^2+ C_2\frac{\Vert
T\Vert_{op}^2}{\Vert
T\Vert_{HS}^2}\lambda_{TX}^2\cr
&\leq&C\lambda_{TX}^2\left(\frac{\Vert T\Vert_{HS}^\frac{4}{3}}{\Vert
T\Vert_{op}^\frac{4}{3}}+\frac{\Vert T\Vert_{op}^2}{\Vert
T\Vert_{HS}^2}\right)\cr
&\leq&C\lambda_{TX}^2\frac{\Vert T\Vert_{HS}^\frac{4}{3}}{\Vert
T\Vert_{op}^\frac{4}{3}}\leq Cn^\frac{2}{3}\lambda_{TX}^2,
\end{eqnarray*}
since $\Vert T\Vert_{op}\leq\Vert T\Vert_{HS}\leq\sqrt n\Vert T\Vert_{op}$.
\end{proof}

The square negative correlation property appeared in \cite{ABP} in the context
of the central limit problem for convex bodies.

\begin{df}
 Let $X$ be a centered log-concave random vector in $\R^n$ and
$\{\eta_i\}_{i=1}^n$ an orthonormal basis of $\R^n$. We say that $X$ satisfies
the square negative correlation property with respect to $\{\eta_i\}_{i=1}^n$ if
for every $i\neq j$
$$
\E\langle X,\eta_i\rangle^2\langle X,\eta_j\rangle^2\leq\E\langle
X,\eta_i\rangle^2\E\langle X,\eta_j\rangle^2.
$$
\end{df}

In \cite{ABP}, the authors showed that a random vector uniformly distributed on
$B_p^n$ satisfies the square negative correlation property with respect to the
canonical basis of $\R^n$. The same property was proved for random vectors
uniformly distributed on generalized Orlicz balls in \cite{W}, where it was also
shown that this property does not hold in general, even in the class of random
vectors uniformly distributed on 1-symmetric convex bodies.

It is easy to see
that if a random centered log-concave vector $X$ satisfies the
square negative correlation property with respect to some
orthonormal basis,
 then it also
satisfies the Conjecture \ref{variance}. Furthermore,
the following proposition shows that, in such
case, also some class of linear perturbations of $X$ verify the Conjecture \ref{variance}.
\begin{proposition}\label{squnegcorr}
 Let $X$ be a centered log-concave random vector in $\R^n$ satisfying the square
negative correlation property with respect to any orthonormal basis, then
the Conjecture \ref{variance} holds for any linear image $T\,\in GL(n)$.
\end{proposition}
\begin{proof} Let $T=V\Lambda U$, with $U, V\in O(n)$ and $\Lambda=[\lambda_1,\dots,
\lambda_n]$ ( $\lambda_i>0$) a diagonal map. Let $\{\eta_i\}_i$ the
orthonormal basis defined by $U\eta_i=e_i$ for all $i$.
By the square negative correlation property
\begin{eqnarray*}
\text{Var } |TX|^2&=&\sum_{i=1}^n\lambda_i^4(\E\langle
X,\eta_i\rangle^4-(\E\langle X,\eta_i\rangle^2)^2)\cr
&+&\sum_{i\neq j}\lambda_i^2\lambda_j^2(\E\langle X,\eta_i\rangle^2\langle
X,\eta_j\rangle^2-\E\langle X,\eta_i\rangle^2\E\langle X,\eta_j\rangle^2)\cr
&\leq&\sum_{i=1}^n\lambda_i^4(\E\langle
X,\eta_i\rangle^4-(\E\langle X,\eta_i\rangle^2)^2)\cr
\end{eqnarray*}
By Borell's lemma (see, for instance,
 \cite{Bor}, Lemma 3.1 or \cite{MS}, Appendix III)
 $$\E\langle
X,\eta_i\rangle^4\leq C(\E\langle X,\eta_i\rangle^2)^2.$$ Thus
\begin{eqnarray*}
\text{Var } |TX|^2&\leq&C\sum_{i=1}^n\lambda_i^4(\E\langle X,\eta_i\rangle^2)^2\leq
C\lambda_{TX}^2\sum_{i=1}^n\lambda_i^2\E\langle X,\eta_i\rangle^2\cr
&=&C\lambda_{TX}^2\E|TX|^2.
\end{eqnarray*}
\end{proof}

\begin{rmk}
Notice that if $X$ satisfies the square
negative correlation property with respect to one orthonormal
 basis $\{\eta_i\}_{i=1}^n$ and $U$ is the orthogonal map such
that $U(\eta_i)=e_i$, the same proof gives that $\Lambda U X$
 verifies Conjecture \ref{variance} for any linear image $\Lambda=[\lambda_1,\dots,
\lambda_n]$ ( $\lambda_i>0$).
\end{rmk}

Even though verifying the variance conjecture is not equivalent to satisfy a
square negative correlation property, the following lemma shows that it is
equivalent to satisfy a ``weak averaged square negative correlation'' property
with respect to one and every orthonormal basis.

\begin{lemma}
 Let $X$ be a centered log concave random vector in $\R^n$. The following are
equivalent
\begin{itemize}
 \item [i)]$X$ verifies the variance conjecture with constant $C_1$
$$
\text{Var }|X|^2\leq C_1\lambda_X^2\E|X|^2.
$$
\item[ii)]$X$ satisfies the following ``weak averaged square negative
correlation'' property with respect to some ortonomal basis $\{\eta_i\}_{i=1}^n$
with constant $C_2$
$$
\sum_{i\neq j}(\E\langle X,\eta_i\rangle^2\langle
X,\eta_j\rangle^2-\E\langle X,\eta_i\rangle^2\E\langle X,\eta_j\rangle^2)\leq
C_2\lambda_X^2\E|X|^2.
$$
\item[iii)]$X$ satisfies the following ``weak averaged square negative
correlation'' property with respect to every ortonomal basis
$\{\eta_i\}_{i=1}^n$ with constant $C_3$
$$
\sum_{i\neq j}(\E\langle X,\eta_i\rangle^2\langle
X,\eta_j\rangle^2-\E\langle X,\eta_i\rangle^2\E\langle X,\eta_j\rangle^2)\leq
C_3\lambda_X^2\E|X|^2,
$$
\end{itemize}
where
$$
C_2\leq C_1\leq C_2+C \hspace{1cm} C_3\leq C_1\leq C_3+C
$$
with $C$ an absolute constant.
\end{lemma}

\begin{proof}
 For any orthonormal basis $\{\eta_i\}_{i=1}^n$ we have
\begin{eqnarray*}
\text{Var } |X|^2&=&\sum_{i=1}^n(\E\langle
X,\eta_i\rangle^4-(\E\langle X,\eta_i\rangle^2)^2)\cr
&+&\sum_{i\neq j}(\E\langle X,\eta_i\rangle^2\langle
X,\eta_j\rangle^2-\E\langle X,\eta_i\rangle^2\E\langle X,\eta_j\rangle^2).\cr
\end{eqnarray*}
Denoting by $A(\eta)$ the second term we have, using Borell's lemma, that
$$
A(\eta)\leq \text{Var }|X|^2\leq C\lambda_X^2\E|X|^2+ A(\eta),
$$since
$\displaystyle\sum_{i=1}^n\E\langle
X,\eta_i\rangle^4\leq C\sup_i\E\langle
X,\eta_i\rangle^2
\sum_{i=1}^n\E\langle
X,\eta_i\rangle^2= C\lambda_X^2\E|X|^2.
$
\end{proof}

%%%%%%%%%%%%%%%%%%%%%%%%%%%%%%%%%%%%%%%%%%%%%%%%%%%%%%%%%%%%%%%%%%%%%%%%%%%%%%%%
\section{Hyperplane projections of the cross-polytope and the cube}\label{projections}

In this section we are going to give some new examples of random vectors
verifying the variance conjecture. We will consider the family of random vectors
uniformly distributed on a hyperplane projection of some symmetric isotropic
convex body $K_0$. These random vectors will not necessarily be isotropic.
However, as we will see in the next proposition, they will be almost isotropic.
{\it i.e.} the spectral condition number $B$ of their covariance matrix verifies
$1\leq B\leq C$ for some absolute constant $C$.

\begin{proposition}{\label{sigma}}
Let $K_0\subset\R^n$ be a symmetric isotropic convex body, and let
$H=\theta^\perp$ be
any
hyperplane. Let $K=P_H(K_0)$ and $X$ a random vector uniformly distributed on
$K$.
 Then, for any $\xi\in S_H=\{x\in H; |x|=1\}$
\[
\E\langle X,\xi\rangle^2\sim\frac{1}{\Vol(K)^{1+\frac{2}{n-1}}}\int_{K}
\langle x,\xi\rangle^2dx\sim
L_{K_0}^2.\]
Consequently $\lambda_X\sim L_{K_0}$ and $B(X)\sim 1$.
\end{proposition}

\begin{proof}
The two first expressions are equivalent, since $\Vol(K)^\frac{1}{n-1}\sim 1$.
Indeed, using Hensley's result \cite{He} and the best general known upper bound
for the isotropy constant of an $n$-dimensional convex body \cite{K1}, we have
$$\Vol(K)^\frac{1}{n-1}\geq\Vol(K_0\cap H)^\frac{1}{n-1}\geq
\left(\frac{c}{L_{K_0}}\right)^\frac{1}{n-1}\geq
\left(\frac{c}{n^\frac{1}{4}}\right)^\frac{1}{n-1}\geq c.$$

On the other hand, since (see \cite{RS} for a proof)
$$\frac{1}{n}\Vol(K)\Vol(K_0\cap H^\perp)\leq \Vol(K_0)=1$$
we have
$$
\Vol(K)^\frac{1}{n-1}\leq\left(\frac{n}{2r(K_0)}\right)^\frac{1}{n-1}
\leq\left(\frac{n}{2L_{K_0}}\right)^\frac{1}{n-1}\leq(cn)^\frac{1}{n-1}\leq c,
$$
where we have used that $r(K_0)=\sup\{r\,:\, rB_2^n\subseteq K_0\}\geq
cL_{K_0}$, see \cite{KLS}.

Let us prove the last estimate. Let $S(K_0)$ be the Steiner symmetrization of
$K_0$
with respect to the hyperplane $H$ and let $S_1$ be its isotropic position. It
is known (see \cite{B} or \cite{MP}) that for any isotropic $n$-dimensional
convex body $L$ and any linear
subspace $E$ of codimension $k$
$$
\Vol(L\cap E)^\frac{1}{k}\sim \frac{L_C}{L_L},
$$
where $C$ is a convex body in $E^\perp$. In particular, we have that
$$
\Vol(S_1\cap H)\sim \frac{1}{L_{S(K_0)}}
$$
and
$$
\Vol(S_1\cap H\cap\xi^\perp)\sim\frac{1}{L_{S(K_0)}^2}.
$$

Since $K_0$ is symmetric, $S_1\cap H$ is symmetric and thus centered. Then, by
Hensley's result \cite{He}
\begin{align*}
 L_{S(K_0)}^2&\sim\frac{\Vol(S_1\cap H)^2}{\Vol(S_1\cap
H\cap\xi^\perp)^2}\sim\frac{1}{\Vol(S_1\cap H)}\int_{S_1\cap H}\langle
x,\xi\rangle^2dx\\
&\sim\frac{1}{\Vol(S_1\cap H)^{1+\frac{2}{n-1}}}\int_{S_1\cap
H}\langle x,\xi\rangle^2dx,
\end{align*}
because $\Vol(S_1\cap H)\sim\frac{1}{L_{S(K_0)}}$ and so $\Vol(S_1\cap
H)^\frac{1}{n-1}\sim c$.

But now, $\widetilde{S_1\cap H}=\Vol(S_1\cap H)^{-\frac{1}{n-1}}(S_1\cap
H)=\widetilde{S(K_0)\cap H}=\widetilde{K}$, because, even though $S(K_0)$ is
not isotropic, $S_1$ is obtained from $S(K_0)$ multiplying it by some $\lambda$
in
$H$ and by $\frac{1}{\lambda^\frac{1}{n-1}}$ in $H^\perp$. Thus,
$$
L_{S(K_0)}^2\sim\int_{\widetilde{S_1\cap H}}\langle
x,\xi\rangle^2dx=\int_{\widetilde{K}}\langle
x,\xi\rangle^2dx=\frac{1}{\Vol(K)^{1+\frac{2}{n-1}}}\int_{K}\langle
x,\xi\rangle^2dx
$$
and since $L_{S(K_0)}\sim L_{K_0}$ we obtain the result.
\end{proof}

The first examples we consider are the random vectors uniformly distributed on
hyperplane projections of the cube. We will see that these random vectors
satisfy the negative square correlation property with respect to any orthonormal
basis. Consequently, by Proposition \ref{squnegcorr}, any linear image of
these random vectors will verify the
variance conjecture with an absolute constant.

\begin{thm}
Let $\theta\in S^{n-1}$ and let $K=P_HB_\infty^n$ be the projection of
$B_\infty^n$
on the hyperplane $H=\theta^\perp$. If $X$ is a random vector uniformly
distributed on $K$ then, for
any two orthonormal vectors $\eta_1,\eta_2\in H$, we have
$$
\E\langle X,\eta_1\rangle^2\langle
X,\eta_2\rangle^2\leq\E\langle X, \eta_1\rangle^2
\E\langle X, \eta_2\rangle^2.
$$
Consequently, $X$ satisfies the negative square correlation property with
respect to any orthonormal basis in $H$.
\end{thm}

\begin{proof}
 Let $F_i$ denote the facet $F_i=\{y \in B_\infty^n; y_{|i|}=\signum{i}\}$ ,
$i\in
\{\pm1,\dots, \pm n\}$. From Cauchy's formula, it is clear that for any function
$f$
\begin{align*}
 \E f(X)&=\sum_{i=\pm1}^{\pm n}\frac{|\theta_{|i|}|}{2\Vert \theta\Vert_1}
\E(f(P_H Y^i))
\end{align*}
 where $Y^i$ is a random vector uniformly distributed on the facet $F_i$.

Remark that $$\Vol(P_H(F_i))=|\langle \theta,
e_{|i|}\rangle|\,\Vol(F_i)=2^{n-1}|\theta_{|i|}|$$
for $i=\pm1,\dots,\pm n$ and $\Vol(K)=2^{n-1}\Vert \theta\Vert_1$.

For any unit vector $\eta\in H$, we have by isotropicity of the facets of
$B_\infty^n$,
\begin{align*}
 \E\langle X,\eta\rangle^2&=\sum_{i=\pm1}^{\pm
n}\frac{|\theta_{|i|}|}{2\Vert\theta\Vert_1}
\E\langle Y^i,\eta\rangle^2
=\sum_{i=\pm1}^{\pm n}\frac{|\theta_{|i|}|}{2\Vert\theta\Vert_1}
\E\sum_{j=1}^n\eta_j^2{Y_j^i}^2\\
&=\frac{1}{2}\sum_{j=1}^n\eta_j^2\sum_{i=\pm1}^{\pm
n}\frac{|\theta_{|i|}|}{\Vert\theta\Vert_1} \E {Y^i_j}^2
=\sum_{j=1}^n\eta_j^2\left(\frac{|\theta_j|}{\Vert\theta\Vert_1}+\frac{1}{3}
\sum_{i\not=
j}\frac{|\theta_i|}{\Vert\theta\Vert_1} \right)\\
&=\sum_{j=1}^n\eta_j^2\left(\frac{2|\theta_j|}{3\Vert\theta\Vert_1}+\frac{1}{3}
\right)
=\frac {1}{3}+\frac{2}{3}
\sum_{j=1}^n\eta_j^2\frac{|\theta_j|}{\Vert\theta\Vert_1}.
\end{align*}

Consequently,
\begin{eqnarray*}
&&\E\langle X, \eta_1\rangle^2\E\langle X, \eta_2\rangle^2\cr
&=&\frac{1}{9}+\frac{2}{9}\sum_{i=1}^n\frac{|\theta_i|}{\Vert\theta\Vert_1}
(\eta_1(i)^2+\eta_2(i)^2)+\frac{4}{9}\sum_{i_1,i_2=1}^n\frac{|\theta_{i_1}
||\theta_{i_2}|}{\Vert\theta\Vert_1^2}\eta_1(i_1)^2\eta_2(i_2)^2\cr
&\geq&\frac{1}{9}+\frac{2}{9}\sum_{i=1}^n\frac{|\theta_i|}{\Vert\theta\Vert_1}
(\eta_1(i)^2+\eta_2(i)^2).
\end{eqnarray*}

On the other hand, by symmetry,
\begin{eqnarray*}
&&\E\langle X,\eta_1\rangle^2\langle
X,\eta_2\rangle^2=\sum_{i=\pm1}^{\pm
n}\frac{|\theta_{|i|}|}{2\Vert\theta\Vert_1}\E\langle Y^i,\eta_1\rangle^2\langle
Y^i,\eta_2\rangle^2\cr
&=&\sum_{i=\pm1}^{\pm n}\frac{|\theta_{|i|}|}{2\Vert\theta\Vert_1}
\frac{1}{\Vol(B_\infty^{n-1})}\int_{B_\infty^{n-1}}(\langle
y,P_{e_{|i|}}\eta_1\rangle +\signum(i)\eta_1(i))^2
(\langle y,P_{e_{|i|}}\eta_2\rangle +\signum(i)\eta_2(i))^2dy\cr
&=&\sum_{i=1}^n\frac{|\theta_i|}{\Vert\theta\Vert_1}\frac{1}{\Vol(B_\infty^{n-1}
)}\int_{B_\infty^{n-1}}\left(\langle y,P_{e_i^\perp}\eta_1\rangle^2\langle
y,P_{e_i^\perp}\eta_2\rangle^2+\eta_2(i)^2\langle
y,P_{e_i^\perp}\eta_1\rangle^2+\right.\cr
&+&\left.\eta_1(i)^2\langle
y,P_{e_i^\perp}\eta_2\rangle^2+\eta_1(i)^2\eta_2(i)^2+4\eta_1(i)\eta_2(i)\langle
y,P_{e_i^\perp}\eta_1\rangle\langle y,P_{e_i^\perp}\eta_2\rangle\right) dy\cr
&=&\sum_{i=1}^n\frac{|\theta_i|}{\Vert\theta\Vert_1}\left(\frac{1}{3}
\eta_1(i)^2|P_{e_i^\perp}\eta_2|^2+\frac{1}{3}\eta_2(i)^2|P_{e_i^\perp}
\eta_1|^2+\eta_1(i)^2\eta_2(i)^2+\right.\cr
&+&4\eta_1(i)\eta_2(i)\frac{1}{\Vol(B_\infty^{n-1})}\int_{B_\infty^{n-1}}\langle
y,P_{e_i^\perp}\eta_1\rangle\langle y,P_{e_i^\perp}\eta_2\rangle
dy\cr
&+&\left.\frac{1}{\Vol(B_\infty^{n-1})}\int_{B_\infty^{n-1}}\langle
y,P_{e_i^\perp}\eta_1\rangle^2\langle y,P_{e_i^\perp}\eta_2\rangle^2
dy\right)\cr
&=&\sum_{i=1}^n\frac{|\theta_i|}{\Vert\theta\Vert_1}\left(\frac{1}{3}
\eta_1(i)^2+\frac{1}{3}\eta_2(i)^2+\frac{1}{3}\eta_1(i)^2\eta_2(i)^2+\right.\cr
&+&4\eta_1(i)\eta_2(i)\frac{1}{\Vol(B_\infty^{n-1})}\int_{B_\infty^{n-1}}\langle
y,P_{e_i^\perp}\eta_1\rangle\langle y,P_{e_i^\perp}\eta_2\rangle
dy\cr
&+&\left.
\frac{1}{\Vol(B_\infty^{n-1})}\int_{B_\infty^{n-1}}\langle
y,P_{e_i^\perp}\eta_1\rangle^2\langle y,P_{e_i^\perp}\eta_2\rangle^2
dy\right)\cr
\end{eqnarray*}

Since
\begin{eqnarray*}
\frac{1}{\Vol(B_\infty^{n-1})}\int_{B_\infty^{n-1}}
\langle
y,P_{e_i^\perp}\eta_1\rangle\langle y,P_{e_i^\perp}\eta_2\rangle
dy&=&\frac{1}{\Vol(B_\infty^{n-1})}\int_{B_\infty^{n-1}}\left(\sum_{l_1,l_2\neq
i}y_{l_1}y_{l_2}\eta_1(l_1)\eta_2(l_2)\right)dy\cr
&=&\frac{1}{\Vol(B_\infty^{n-1})}\int_{B_\infty^{n-1}}\left(\sum_{l\neq
i}y_l^2\eta_1(l)\eta_2(l)\right)dy\cr
&=&\frac{1}{3}\sum_{l\neq i}\eta_1(l)\eta_2(l)\cr
&=&\frac{1}{3}(\langle \eta_1,\eta_2\rangle-\eta_1(i)\eta_2(i))\cr
&=&-\frac{1}{3}\eta_1(i)\eta_2(i)
\end{eqnarray*}
the previous sum equals
$$
\sum_{i=1}^n\frac{|\theta_i|}{\Vert\theta\Vert_1}\left(\frac{1}{3}
\eta_1(i)^2+\frac{1}{3}\eta_2(i)^2-\eta_1(i)^2\eta_2(i)^2+
\frac{1}{\Vol(B_\infty^{n-1})}\int_{B_\infty^{n-1}}\langle
y,P_{e_i^\perp}\eta_1\rangle^2\langle
y,P_{e_i^\perp}\eta_2\rangle^2 dy\right).
$$

Now,
\begin{eqnarray*}
& &\frac{1}{\Vol(B_\infty^{n-1})}\int_{B_\infty^{n-1}}\langle
y,P_{e_i^\perp}\eta_1\rangle^2\langle y,P_{e_i^\perp}\eta_2\rangle^2 dy\cr
&=&\frac{1}{\Vol(B_\infty^{n-1})}\int_{B_\infty^{n-1}}\left(\sum_{l_1,l_2,l_3,
l_4}y_
{l_1}y_{l_2}y_{l_3}y_{l_4}
\eta_1(l_1)\eta_1(l_2)\eta_2(l_3)\eta_2(l_4)\right)dy\cr
&=&\sum_{l\neq
i}\eta_1(l)^2\eta_2(l)^2\frac{1}{\Vol(B_\infty^{n-1})}\int_{B_\infty^{n-1}}
y_l^4dy\cr
&+&\sum_{l_1\neq l_2 (\neq
i)}\frac{1}{\Vol(B_\infty^{n-1})}\int_{B_\infty^{n-1}}y_{l_1}^2y_{l_2}
^2dy(\eta_1(l_1)^2\eta_2(l_2)^2+\eta_1(l_1)\eta_1(l_2)\eta_2(l_1)\eta_2(l_2))\cr
&=&\frac{1}{5}\sum_{l\neq i}\eta_1(l)^2\eta_2(l)^2+\frac{1}{9}\sum_{l_1\neq l_2
(\neq
i)}(\eta_1(l_1)^2\eta_2(l_2)^2+\eta_1(l_1)\eta_1(l_2)\eta_2(l_1)\eta_2(l_2))\cr
&=&\frac{1}{5}\sum_{l\neq i}\eta_1(l)^2\eta_2(l)^2\cr
&+&\frac{1}{9}\left[\sum_{l\neq
i}\left(\eta_1(l)^2(1-\eta_2(l)^2-\eta_2(i)^2)+\eta_1(l)\eta_2(l)(\langle
\eta_1,\eta_2\rangle-\eta_1(l)\eta_2(l)-\eta_1(i)\eta_2(i))\right)\right]\cr
&=&\frac{1}{5}\sum_{l\neq i}\eta_1(l)^2\eta_2(l)^2\cr
&+&\frac{1}{9}\left(1-\eta_1(i)^2-\sum_{l\neq
i}\eta_1(l)^2\eta_2(l)^2-\eta_2(i)^2+\eta_1(i)^2\eta_2(i)^2\right.\cr
&-&\left.\sum_{l\neq
i}\eta_1(l)^2\eta_2(l)^2+\eta_1(i)^2\eta_2(i)^2\right)\cr
&=&\frac{1}{9}-\frac{1}{9}\eta_1(i)^2-\frac{1}{9}\eta_2(i)^2+\frac{2}{9}
\eta_1(i)^2\eta_2(i)^2-\frac{1}{45}\sum_{l\neq i}\eta_1(l)^2\eta_2(l)^2.
\end{eqnarray*}

Consequently
\begin{eqnarray*}
\E\langle X,\eta_1\rangle^2\langle
X,\eta_2\rangle^2&=&\frac{1}{9}+\sum_{i=1}^n\frac{|\theta_i|}{\Vert\theta\Vert_1
}\left(\frac{2}
{9}\eta_1(i)^2+\frac{2}{9}\eta_2(i)^2-\frac{7}{9}
\eta_1(i)^2\eta_2(i)^2\right.\cr
&-&\left.\frac{1}{45}
\sum_{l\neq i}\eta_1(i)^2\eta_2(l)^2\right)\cr
&\leq&\frac{1}{9}+\frac{2}{9}\sum_{i=1}^n\frac{|\theta_i|}{\Vert\theta\Vert_1}
(\eta_1(i)^2+\eta_2(i)^2)
\end{eqnarray*}
which concludes the proof.
\end{proof}
%%%%%%%%%%%%%%%%%%%%%%%%%%%%%%%%%%%%%%%%%%%%%%%%%%%%%%%%%%%%%%%%%%%%%%%%%%%%%%%%
By Proposition \ref{squnegcorr} we obtain the following
\begin{cor}
There exists an absolute constant $C$ such that for every hyperplane $H$
and any linear map $T$, if $X$ is a random vector uniformly distributed on
$P_H B_\infty^n$, then $TX$ verifies the variance conjecture with constant
$C$, {\it i.e.}
$$
\text{Var }|TX|^2\leq C\lambda_{TX}^2\E|TX|^2
$$
\end{cor}

The next examples we consider are random vectors uniformly distributed on
projections of $B_1^n$. Even though in this case we are not able to prove that
these vectors satify a square negative correlation property, we are still able
to show that they verify the variance conjecture with some absolute constant.

\begin{thm}
 There exists an absolute constant $C$ such that for every hyperplane $H$, if
$X$ is a random vector uniformly distributed on $P_H B_1^n$, $X$ verifies the
variance conjecture with constant $C$, {\it i.e.}
$$
\text{Var } |X|^2\leq C\lambda_X^2\E|X|^2.
$$
\end{thm}

\begin{proof}
First of all, notice that by Proposition \ref{sigma} we have that for every
$\xi\in S^{n-1}\cap H$
$$
\E\langle \textrm{Vol}(B_1^n)^{-\frac{1}{n}}X,\xi\rangle^2\sim L_{B_1^n}^2\sim 1
$$
and so
$$
\lambda_X^2\sim\frac{1}{n^2}\hspace{1cm}\textrm{ and
}\hspace{1cm}\E|X|^2\sim\frac{1}{n}.
$$
Thus, we have to prove that $\displaystyle{\text{Var }|X|^2\leq\frac{C}{n^3}}$.

By Cauchy formula, denoting by $\theta$ the unit vector orthogonal to $H$, $Y$ a
random vector uniformly distributed on $\Delta_{n-1}=\{y\in\R^n\,:\,y_i\geq 0,
\sum_{i=1}^n y_i=1\}$, $\varepsilon$ a random vector, independent of $Y$,  in
$\{-1,1\}^n$ distributed according to
$$\Pro(\varepsilon=\varepsilon_0)=\frac{|\langle\varepsilon_0,\theta\rangle|}{
\sum_{\varepsilon\in\{-1,1\}^n}|\langle\varepsilon,\theta\rangle|}=\frac{\textrm
{Vol}_{n-1}(\Delta_{n-1})|\langle\varepsilon_0,\theta\rangle|}{2\sqrt n
\textrm{Vol}_{n-1}(P_H(B_1^n))}$$
and
$$
\varepsilon x=(\varepsilon_1x_1,\dots,\varepsilon_n x_n)
$$
we have that
\begin{eqnarray*}
\text{Var }|X|^2&=&\E|X|^4-(\E|X|^2)^2=\E|P_H(\varepsilon
Y)|^4-(\E|P_H(\varepsilon
Y)|^2)^2\cr
&=&\E(|\varepsilon Y|^2-\langle\varepsilon Y,\theta\rangle^2)^2-(\E(|\varepsilon
Y|^2-\langle\varepsilon Y,\theta\rangle^2))^2\cr
&=&\E(|Y|^2-\langle Y,\varepsilon\theta\rangle^2)^2-(\E(|Y|^2-\langle
Y,\varepsilon\theta\rangle^2))^2\cr
&\leq&\E|Y|^4+\E\langle \varepsilon Y,\theta\rangle^4-(\E|Y|^2-\E\langle
\varepsilon Y,\theta\rangle^2)^2.
\end{eqnarray*}
Since for every $a,b\in \N$ with $a+b=4$ we have
$$
\E Y_1^aY_2^b=\frac{a!b!}{(n+3)(n+2)(n+1)n}
$$
we have
\begin{eqnarray*}
\E|Y|^4&=&n\E Y_1^4+n(n-1)\E Y_1^2Y_2^2\cr
&=&\frac{4!}{(n+3)(n+2)(n+1)} +\frac{4(n-1)}{(n+3)(n+2)(n+1)}\cr
&=&\frac{4}{n^2}+O\left(\frac{1}{n^3}\right).
\end{eqnarray*}
Denoting by $\epsilon$ a radom vector uniforly distributed on $\{-1,1\}^n$ we
have, by Khintchine inequality,
\begin{eqnarray*}
\E\langle \varepsilon
Y,\theta\rangle^4&=&\frac{\textrm{Vol}_{n-1}(\Delta_{n-1})}{2\sqrt n
\textrm{Vol}_{n-1}(P_H(B_1^n))}\E_Y\sum_{\varepsilon\in\{-1,1\}^n}|\langle
\varepsilon,\theta\rangle|\langle \varepsilon
Y,\theta\rangle^4\cr
&=&\frac{2^n\textrm{Vol}_{n-1}(\Delta_{n-1})}{2\sqrt n
\textrm{Vol}_{n-1}(P_H(B_1^n))}\E_Y\E_{\epsilon}|\langle
\epsilon,\theta\rangle|\langle \epsilon
Y,\theta\rangle^4\cr
&\leq& C\E_Y\left(\E_\epsilon\langle
\epsilon,\theta\rangle^2\right)^\frac{1}{2}\left(\E_\epsilon\langle \epsilon
Y,\theta\rangle^8\right)^\frac{1}{2}\cr
&\leq&C\E_Y\left(\sum_{i=1}^nY_i^2\theta_i^2\right)^2\cr
&=&C\left(\E Y_1^4\sum_{i=1}^n\theta_i^4+\E Y_1^2 Y_2^2\sum_{i\neq
j}\theta_i^2\theta_j^2\right)\cr
&=&\frac{C}{(n+3)(n+2)(n+1)n}\left(24\sum_{i=1}^n\theta_i^4+4\sum_{i,j=1}
^n\theta_i^2\theta_j^2\right)\cr
&\leq&\frac{C}{n^4}
\end{eqnarray*}
since $\sum_{i=1}^n\theta_i^4\leq\sum_{i=1}^n\theta_i^2=1$.

On the other hand, since
$$
\E Y_1^2=\frac{2}{(n+1)n}\hspace{1cm}\textrm{ and }\hspace{1cm}\E
Y_1Y_2=\frac{1}{(n+1)n}
$$
we have
$$
\E|Y|^2=n\E Y_1^2=\frac{2}{n+1}
$$
and
\begin{eqnarray*}
 \E\langle\varepsilon Y,\theta\rangle^2&=&\E\left(\sum_{i=1}^n
Y_i^2\theta_i^2+\sum_{i\neq
j}\varepsilon_i\varepsilon_jY_iY_j\theta_i\theta_j\right)\cr
&=&\E Y_1^2+\sum_{i\neq
j}\theta_i\theta_j\E_\varepsilon\varepsilon_i\varepsilon_j\E_Y Y_1Y_2\cr
&=&\frac{2}{(n+1)n}+\frac{1}{(n+1)n}(\E_{\varepsilon}\langle\varepsilon,
\theta\rangle^2-1)\cr
&=&\frac{1}{(n+1)n}(1+\E_{\varepsilon}\langle\varepsilon,
\theta\rangle^2)\sim\frac{1}{n^2},\cr
\end{eqnarray*}
since, by Khintchine inequality,
\begin{eqnarray*}
\E_{\varepsilon}\langle\varepsilon,\theta\rangle^2&=&\frac{\textrm{Vol}_{n-1}
(\Delta_{n-1})}{2\sqrt n
\textrm{Vol}_{n-1}(P_H(B_1^n))}\sum_{\varepsilon\in\{-1,1\}^n}|\langle
\varepsilon,\theta\rangle|^3\cr
&\leq&C\E_{\epsilon}|\langle
\varepsilon,\theta\rangle|^3\sim C
\end{eqnarray*}

Thus
$$
(\E|Y|^2-\E\langle\varepsilon
Y,\theta\rangle^2)^2=\frac{4}{n^2}+O\left(\frac{1}{n^3}\right)
$$ and so
$$
\text{Var }|X|^2\leq \frac{C}{n^3}.
$$
\end{proof}

\proof[Acknowledgements]
Part of this work was done while the first named author was a postdoctoral fellow at the
Department of Mathematical and Statistical Sciences at University of
Alberta. He would like to thank the department for providing such good
environment and working conditions.

The authors also want to thank the referee for providing several comments
which improved the presentation of the paper a lot.

\end{document}